\documentclass{amsart}

\usepackage{verbatim}
\usepackage{amssymb}
\usepackage{amscd}
\usepackage{amsmath}
\usepackage{setspace}

\newtheorem{theorem}{Theorem}[section]

\newtheorem{corollary}[theorem]{Corollary}
\newtheorem{lemma}[theorem]{Lemma}

\numberwithin{equation}{section}

\theoremstyle{remark}

\newtheorem{remark}[theorem]{Remark}

\renewcommand{\Re}{\mathrm{Re}}

\newcommand{\Z}{\mathbb Z}

\newcommand{\A}{\mathbb A}

\newcommand{\Sp}{\mathrm{Sp}}
\newcommand{\GSp}{\mathrm{GSp}}

\title[Prime density results for Hecke eigenvalues]{Prime density results for Hecke eigenvalues of a Siegel cusp form}
\author{Abhishek Saha}

 \address{ETH Z\"urich -- D-MATH\\
  R\"amistrasse 101\\
  8092 Z\"urich\\
  Switzerland}

\email{abhishek.saha@math.ethz.ch}
\begin{document}
\begin{abstract}Let $F\in S_k(\Sp(2g, \Z))$ be a cuspidal Siegel eigenform of genus $g$ with normalized Hecke eigenvalues $\mu_F(n)$. Suppose that the associated automorphic representation $\pi_F$ is locally tempered everywhere. For each $c>0$ we consider the set of primes $p$ for which $|\mu_F(p)| \ge c$ and we provide an explicit upper bound on the density of this set. In the case $g=2$, we also provide an explicit upper bound on the density of the set of primes $p$ for which $\mu_F(p) \ge c$.  \end{abstract}
\bibliographystyle{plain}
\maketitle
\section{Introduction}
For a classical Hecke eigenform $f(z)$, the recently proved (see~\cite{satotate}) Sato-Tate conjecture asserts that the normalized Hecke eigenvalues $\mu_f(p)$ are equidistributed in $[-2,2]$ with respect to the Sato-Tate measure. For a cuspidal Siegel eigenform $F(Z)$ of genus $g > 1$, no such equidistribution result is currently available. The purpose of this short paper is to prove a couple of results about the distribution of its Hecke eigenvalues using analyticity properties of $L$-functions associated to $F$ that are known at present.

Let $F(Z) \in S_k(\Sp(2g, \Z))$ be a cuspidal Siegel eigenform of genus $g$, weight $k$ and Hecke eigenvalues $\lambda_F(n)$, $n>0$. Define the normalized Hecke eigenvalues $\mu_F(n)= \lambda_F(n)/n^{\frac{kg}{2}-\frac{g(g+1)}{4}}$. Further, assume that $F$ satisfies the generalized Ramanujan conjecture, i.e., the associated local representations are all tempered. Then we have  $\mu_F(p) \in [-2^{g}, 2^{g}]$ for all primes $p$.

We prove the following Theorem about the distribution of the quantity $|\mu_F(p)|$.

\begin{theorem}\label{maintheorem1} Suppose that $ g \not \equiv 0 \pmod {4}$. Let $F\in S_k(\Sp(2g, \Z))$ be a cuspidal Siegel eigenform such that the associated automorphic representation $\pi_F$ is locally tempered everywhere. For each $c >0$, let $S^{F,0}_c$ denote the set of primes $p$ for which $|\mu_F(p)| \ge c$. Then we have $$\overline{\delta}_{\mathrm{Dir}}(S^{F,0}_c) \le (2-\tfrac{1}{g})c^{-\frac{2}{g}}$$ where $\overline{\delta}_{\mathrm{Dir}}$ denotes the upper Dirichlet density.
\end{theorem}

See the beginning of Section~\ref{s:proofs} for a remark on the assumption $g \not \equiv 0 \pmod {4}$.

\begin{remark}In the special case $g=2$, the generalized Ramanujan conjecture has been proved by Weissauer~\cite{weissram}. So in this case, we get the unconditional result that $\overline{\delta}_{\mathrm{Dir}}(S^{F,0}_c) \le \frac{3}{2c}$ for all $F$ not in the Maa{\ss} space.\end{remark}

We prove a second result in the case $g=2$, this time about the distribution of the signed quantity $\mu_F(p)$.

\begin{theorem}\label{maintheorem2}Let $F\in S_k(\Sp(4, \Z))$ be a cuspidal eigenform that is not in the Maa{\ss} space. For each $c >0$, let $S^{F,1}_c$ denote the set of primes $p$ for which $\mu_F(p) \ge c$. Then we have $$\overline{\delta}_{\mathrm{Dir}}(S^{F,1}_c) \le \frac{4}{c+4}.$$
\end{theorem}

Since the upper Dirichlet density is always at least as large as the lower natural density $\underline{\delta}_{\mathrm{Nat}}$, Theorems~\ref{maintheorem1} and~\ref{maintheorem2} remain valid with $\overline{\delta}_{\mathrm{Dir}}$ replaced by $\underline{\delta}_{\mathrm{Nat}}$ (see Corollary~\ref{corollary}).

Note that Theorem~\ref{maintheorem1} is non-trivial only in the range $(2-\tfrac{1}{g})^{\frac{g}{2}} \le c \le 2^g$.
Also note that in the range $c\ge\frac{12}{5}$, Theorem~\ref{maintheorem1} automatically implies a stronger version of Theorem~\ref{maintheorem2}! So Theorem~\ref{maintheorem2} is interesting only in the range $0 < c < \frac{12}{5}$.

The proof of Theorem~\ref{maintheorem1} uses the fact that the standard (degree $2g+1$) $L$-function for $F$ has no pole at 1 while the proof of Theorem~\ref{maintheorem2} uses the fact that the spinor (degree $2^g$) $L$-function for $F$ has no pole at 1 (in the case $g=2$). One would expect much stronger results if similar facts can be proved for the higher symmetric powers of these $L$-functions.

Theorems~\ref{maintheorem1} and~\ref{maintheorem2} are essentially assertions about the smallness of the set of primes for which the size of the Hecke eigenvalue is large. In particular, they do not say anything about the proportion (or even infinitude) of primes for which $\mu_F(p)$ is negative and this seems to be a challenging open question. In this context, we note that if one considers the Hecke eigenvalues  $\mu_F(p^j)$ for all positive integers $j$, then it is known in the special case $g=2$ that $\mu_F(p^j)$ takes both positive and negative values infinitely often, see~\cite{kohsign}, \cite{pitschsign}.

\section{Local considerations}\label{slocal}
Let $F(Z) \in S_k(\Sp(2g, \Z))$ be a cuspidal Siegel Hecke eigenform of genus $g$, weight $k$ and eigenvalues $\lambda_F(n)$, $n>0$. Recall that $\mu_F(n)= \lambda_F(n)/n^{\frac{kg}{2}-\frac{g(g+1)}{4}}$ denotes the normalized Hecke eigenvalues.

Let $\pi_F$ be one of the irreducible pieces of the representation of $\GSp(2g,\A)$ generated by the adelization of $F$. For this and other details we refer the reader to~\cite{asgsch}. Then $\pi_F$ is an automorphic representation and is isomorphic to a restricted tensor product $\pi_F =\otimes_v\pi_{F,v}$.

There are two $L$-functions commonly attached to $\pi_F$. The first of these is the \emph{spinor} $L$-function denoted by $L_{\text{spin}}(\pi_{F}, s)$. This has degree $2^g$ and is defined via an Euler product

\begin{equation} \label{spinequationglobal}L_{\text{spin}}(\pi_{F}, s)=\prod_{p<\infty} L_{\text{spin}}(\pi_{F,p}, s)\end{equation}
where the local factors $L_{\text{spin}}(\pi_{F,p}, s)$ are given by

\begin{equation} \label{spinequation} L_{\text{spin}}(\pi_{F,p}, s) = \prod_{k=0}^g \prod_{1 \le i_1 <\ldots <i_k \le g}(1-a_{0,p}a_{i_1,p}\ldots a_{i_k,p} p^{-s})^{-1}
\end{equation}

Here the complex numbers $a_{0,p}$, $a_{1,p}$, $\ldots$, $a_{g,p}$ are the Satake parameters of the local unramified representation $\pi_{F,p}$ associated to $F$ at the place $p$.  Because $F$ has no character, we have the relation
\begin{equation}\label{trivcentral}a_{0,p}^2a_{1,p}\ldots a_{g,p}=1.\end{equation}

 Note that all our Satake parameters and $L$-functions are Langlands normalized, with the functional equation for the $L$-function taking $s$ to $1-s$ .

One also has the \emph{standard} $L$-function $L_{\text{std}}(\pi_{F}, s)$ defined by the Euler product

\begin{equation} \label{stdequationglobal}L_{\text{std}}(\pi_{F}, s)=\prod_{p<\infty} L_{\text{std}}(\pi_{F,p}, s)\end{equation}

where the local factors are given by
\begin{equation} \label{stdequation} L_{\text{std}}(\pi_{F,p}, s) = \bigg((1-p^{-s})\prod_{i=1}^g (1 - a_{i,p}p^{-s})(1 - a_{i,p}^{-1}p^{-s})\bigg)^{-1}
\end{equation}

We say that the associated local representation $\pi_{F,p}$ is \emph{tempered} if $|a_{i,p}|=1$ for $0 \le i \le g$. We say that $\pi_F$ is \emph{locally tempered everywhere} if $\pi_{F,p}$ is tempered for all primes $p$. In that case, the series in~\eqref{spinequationglobal} and~\eqref{stdequationglobal} evidently converge absolutely for $\Re(s)>1.$

By Weissauer's proof~\cite{weissram} of the generalized  Ramanujan conjecture for holomorphic Siegel cusp forms of genus 2, we know that if $F \in S_k(\Sp(4, \Z))$ is not a Saito-Kurokawa lift then $\pi_F$ is locally tempered everywhere.

The relation between the normalized Hecke eigenvalues $\mu_F(p)$ and the Satake parameters is given~\cite{andrianov} as follows: \begin{equation} \label{andrequation} \mu_F(p) = \sum_{k=0}^g\sum_{1 \le i_1 <\ldots <i_k \le g}a_{0,p}a_{i_1,p}\ldots a_{i_k,p}\end{equation}

Using~\eqref{trivcentral} we see that $\mu_F(p)$ is real.

We will need the following two lemmas in the proofs of our main Theorems.

\begin{lemma}\label{oeq}Suppose $\pi_{F,p}$ is tempered. For $\Re(s)>1$ and $r \ge 0$, define the coefficients $m_F(p^r)$ and $\rho_F(p^r)$ via $L_{\text{spin}}(\pi_{F,p}, s) = \sum_{r \ge 0}\frac{m_F(p^r)}{p^{rs}}$ and $L_{\text{std}}(\pi_{F,p}, s) = \sum_{r \ge 0}\frac{\rho_F(p^r)}{p^{rs}}.$
Then $m_F(p^r)$ and $\rho_F(p^r)$ are real numbers and satisfy the inequalities $$|m_F(p^r)| \le \binom{r+2^g -1}{2^g -1}, \qquad |\rho_F(p^r) - \rho_F(p^{r-1})| \le \binom{r+2g-1}{2g-1}.$$
\end{lemma}
\begin{proof} The fact that $m_F(p^r)$ and $\rho_F(p^r)$ are real numbers follows directly from~\eqref{trivcentral} and the assumption that the quantities $a_{i,p}$ lie on the unit circle. Furthermore, it is evident from the definition of the local $L$-factors (see~\eqref{spinequation},~\eqref{stdequation}) that $m_F(p^r)$ is a sum of $\binom{r+2^g -1}{2^g -1}$ terms with absolute value 1 and $\rho_F(p^r)-\rho_F(p^{r-1})$ is a sum of $\binom{r+2g-1}{2g-1}$ terms with absolute value 1: this supplies the desired inequalities.
\end{proof}

\begin{remark}\label{mum}A comparison of~\eqref{spinequation} and~\eqref{andrequation} shows that $\mu_F(p) = m_F(p)$.
\end{remark}

\begin{lemma}\label{lemmaineq}Suppose $\pi_{F,p}$ is tempered and $|\mu_F(p)| \ge c$ for some $c > 0$. Then $$1+\sum_{i=1}^g (a_{i,p} +  a_{i,p}^{-1}) \ge gc^{\frac{2}{g}} -2g +1.$$
\end{lemma}
\begin{proof} Put $U_i = (a_{i,p} +  a_{i,p}^{-1})$ and $U = \sum_{i=1}^g U_i$.

We have \begin{align*}|\mu_F(p)|^2 &= |\prod_{k=0}^g\sum_{1 \le i_1 <\ldots <i_k \le g}a_{i_1,p}\ldots a_{i_k,p}|^2\\ &= \prod_{i=1}^g |1+a_{i,p}|^2\\ &=\prod_{i=1}^g (1+a_{i,p})(1+a_{i,p}^{-1}) \\ &= \prod_{i=1}^g (2 + U_i) \\ & \le (\frac{U}{g} + 2)^g \end{align*} by the AM-GM inequality.

So we have  $$(\frac{U}{g} + 2)^g \ge c^2$$

which is equivalent to $$U \ge  gc^{\frac{2}{g}} -2g$$ as desired.

\end{proof}

\begin{remark} The inequality in the above Lemma is clearly the best possible for $c \le 2^g$ and arbitrary complex numbers $a_{i,p}$ on the unit circle satisfying~\eqref{trivcentral}.
\end{remark}

\section{Two prime density lemmas}

In this section we supply a couple of lemmas on the density of a set of primes. The results are probably classical, but we were unable to find a good reference and so include a proof for completeness.

First, some defintions.
Let $\mathcal{P} = \{2, 3, \ldots \}$ denote the set of primes.

For a subset $S$ of $\mathcal{P}$, the upper Dirichlet density of $S$ is denoted by $\overline{\delta}_{\mathrm{Dir}}(S)$ and is defined by
$$\overline{\delta}_{\mathrm{Dir}}(S) := \limsup_{s \rightarrow 1^+} \frac{\sum_{p\in S} \frac{1}{p^s}}{-\log(s-1)}.$$ This is equivalent to $$\overline{\delta}_{\mathrm{Dir}}(S) := \limsup_{s \rightarrow 1^+} \frac{\sum_{p\in S} \frac{1}{p^s}}{\sum_{p\in \mathcal{P}} \frac{1}{p^s}}.$$

The lower Dirichlet density of $S$ is denoted by $\underline{\delta}_{\mathrm{Dir}}(S)$ and is defined by
$$\underline{\delta}_{\mathrm{Dir}}(S) := \liminf_{s \rightarrow 1^+} \frac{\sum_{p\in S} \frac{1}{p^s}}{-\log(s-1)} = \liminf_{s \rightarrow 1^+} \frac{\sum_{p\in S} \frac{1}{p^s}}{\sum_{p\in \mathcal{P}} \frac{1}{p^s}}.$$

Similarly, we define the upper and lower natural densities by
$$\overline{\delta}_{\mathrm{Nat}}(S) := \limsup_{x \rightarrow \infty}\frac{ \#\{ p: p \le x, p\in S\}}{\#\{ p: p \le x, p\in \mathcal{P}\}} $$ and $$\underline{\delta}_{\mathrm{Nat}}(S) := \liminf_{x \rightarrow \infty}\frac{ \#\{ p: p \le x, p\in S\}}{\#\{ p: p \le x, p\in \mathcal{P}\}}. $$

It is easy to see that $$\overline{\delta}_{\mathrm{Dir}}(S) = 1- \underline{\delta}_{\mathrm{Dir}}(\mathcal{P}-S),$$ $$\overline{\delta}_{\mathrm{Nat}}(S) = 1- \underline{\delta}_{\mathrm{Nat}}(\mathcal{P}-S).$$

\begin{lemma}\label{densityrel}Let $S \subset \mathcal{P}$. Then $$\underline{\delta}_{\mathrm{Nat}}(S)\le \underline{\delta}_{\mathrm{Dir}}(S) \le \overline{\delta}_{\mathrm{Dir}}(S) \le \overline{\delta}_{\mathrm{Nat}}(S)$$
\end{lemma}
\begin{proof}
Let $S(x)$ be the number of primes $p\in S$ such that $p\le x$ and $\pi(x)$ the number of primes less than or equal to $x$. Let $dS(t)$ denote the counting measure for $S$. By partial summation, we have for $s>1$,

\begin{align*}\sum_{p \in S, p\le x} \frac{1}{p^s} &= \int_1^x t^{-s}dS(t)\\&= \frac{S(x)}{x^s} + s\int_1^x\frac{S(t)}{t^{s+1}}dt \\&= s\int_1^x\frac{S(t)}{t^{s+1}}dt + O(1) \end{align*} where the implied constant can be taken to be equal to 1 (and is thus uniform in $x$ and $s$).

By hypothesis, for any $\epsilon >0$ there exists $x_0 = x_0(\epsilon) \ge 1$ such that for all $x > x_0$ we have $$S(x) \le (\overline{\delta}_{\mathrm{Nat}}(S) + \epsilon)\pi(x).$$ Since $s\int_1^{x_0}\frac{S(t)}{t^{s+1}}dt =O(1)$ where the implied constant is uniform for $s$ in $(1,2)$, we have $$\sum_{p \in S} \frac{1}{p^s} \le (\overline{\delta}_{\mathrm{Nat}}(S) + \epsilon)\int_1^\infty\frac{\pi(t)}{t^{s+1}}dt +O(1).$$ By reversing the partial summation argument from above, we have $$\int_1^\infty\frac{\pi(t)}{t^{s+1}}dt = \sum_{p \in \mathcal{P}} \frac{1}{p^s} +O(1).$$ This gives $$ \frac{\sum_{p \in S} \frac{1}{p^s}}{\sum_{p \in \mathcal{P}} \frac{1}{p^s}} = (\overline{\delta}_{\mathrm{Nat}}(S) + \epsilon) +o(1)$$ as $s \rightarrow 1^+$. Thus $\overline{\delta}_{\mathrm{Dir}}(S) \leq \overline{\delta}_{\mathrm{Nat}}(S) + \epsilon$. Letting $\epsilon \rightarrow 0$ we conclude that $$\overline{\delta}_{\mathrm{Dir}}(S)\le\overline{\delta}_{\mathrm{Nat}}(S).$$ Replacing $S$ by $\mathcal{P}-S$ in the above immediately gives $$\underline{\delta}_{\mathrm{Nat}}(S)\le \underline{\delta}_{\mathrm{Dir}}(S) .$$

\end{proof}

\begin{lemma}\label{lemmalfunc}Let $S \subset \mathcal{P}$. Let $C>0$, $E \ge D>0$ be constants and let $f$ be a real-valued function on $\mathcal{P}$ such that the following conditions hold: \begin{enumerate}
\item $-C \le f(p) \le E$ for all $p \in \mathcal{P},$
\item $f(p) \ge D$ for all $p \in S$.
\end{enumerate}
Define the complex analytic function $L(s) := \sum_p \frac{f(p)}{p^s}$ on the region $\Re(s) > 1$, and assume that we have $\limsup_{s \rightarrow 1^+} L(s) < \infty$ as $s$ approaches 1 from the right on the real line. Then,
$$\overline{\delta}_{\mathrm{Dir}}(S)  \le \frac{C}{C+D} $$
\end{lemma}

\begin{proof} We have $$1-\overline{\delta}_{\mathrm{Dir}}(S) = \liminf_{s \rightarrow 1^+} \frac{\sum_{p\in \mathcal{P}-S} \frac{1}{p^s}}{-\log(s-1)}. $$

It follows that \begin{equation}\label{eq2}\overline{\delta}_{\mathrm{Dir}}(S) - 1 = \limsup_{s \rightarrow 1^+} \frac{-\sum_{p\in \mathcal{P}-S} \frac{1}{p^s}}{-\log(s-1)}. \end{equation}

Choose $\epsilon>0, \delta >0$. Then we can find $1< s'<1+\delta$ such that

$$\frac{\sum_{p\in S} \frac{1}{p^{s'}}}{-\log(s'-1)} >  \overline{\delta}_{\mathrm{Dir}}(S) - \epsilon$$
 $$\frac{-\sum_{p\in \mathcal{P}-S} \frac{1}{p^{s'}}}{-\log(s'-1)} > \overline{\delta}_{\mathrm{Dir}}(S) - 1 - \epsilon.$$

So we get \begin{align*} \frac{L(s')}{-\log(s'-1)} &\ge \sum_{p\in S} \frac{\frac{D}{p^{s'}}}{-\log(s'-1)} - \sum_{p\in \mathcal{P}-S} \frac{\frac{C}{p^{s'}}}{-\log(s'-1)}\\ &> D(\overline{\delta}_{\mathrm{Dir}}(S) - \epsilon) + C(\overline{\delta}_{\mathrm{Dir}}(S) - 1 - \epsilon)\end{align*}

Now let $\delta \rightarrow 0$. Then because $\limsup_{s \rightarrow 1^+} L(s) < \infty$ we have $\limsup_{s' \rightarrow 1^+}\frac{L(s')}{-\log(s'-1)} \le 0.$ Thus $$ D(\overline{\delta}_{\mathrm{Dir}}(S) - \epsilon) + C(\overline{\delta}_{\mathrm{Dir}}(S) - 1 - \epsilon) < 0$$ for all $\epsilon>0$, which implies that

$$\overline{\delta}_{\mathrm{Dir}}(S) \le \frac{C}{C+D}.$$

\end{proof}

\section{Proofs of the main results}\label{s:proofs}
Let $F$ be as in Section~\ref{slocal}. Suppose $\pi_F$ is locally tempered everywhere. Then $L_{\text{spin}}(\pi_{F}, s)$ and $L_{\text{std}}(\pi_{F}, s)$ converge absolutely for $\Re(s)>1$ and are non-vanishing in that region. Also, we have the following facts.

\begin{theorem}[Mizumoto~\cite{miz}] Suppose $g \not \equiv 0 \pmod{4}$. Then $L_{\text{std}}(\pi_{F}, s)$ can be continued to a meromorphic function on the entire complex plane that has no pole at $s=1$.
\end{theorem}

\begin{remark} Actually the condition $g \not \equiv 0 \pmod{4}$ can be replaced by the weaker condition that $F$ is not in the subspace generated by certain theta series.
\end{remark}

\begin{theorem}[Andrianov~\cite{andrianov}] Suppose $g=2$ and $F$ is not a Saito-Kurokawa lift. Then $L_{\text{spin}}(\pi_{F}, s)$ can be continued to a meromorphic function on the entire complex plane that has no pole at $s=1$.
\end{theorem}

It is now a simple matter to conclude the proofs of our main Theorems.

\begin{proof}[Proof of Theorem~\ref{maintheorem1}] Consider the function $\log L_{\text{std}}(\pi_{F}, s)$ on the region $\Re(s)>1$. Using the temperedness of $\pi_F$, the Euler product~\eqref{stdequationglobal}, and Lemma~\ref{oeq}, we have

\begin{align*}\log L_{\text{std}}(\pi_{F}, s) &= \sum_{p \in \mathcal{P}}\left(\frac{\rho_F(p)}{p^s} + O(p^{-2})\right)\\&= \sum_{p \in \mathcal{P}}\frac{\rho_F(p)}{p^s} + O(1) \end{align*}
where $\rho_F(p) = 1+\sum_{i=1}^g (a_{i,p} +  a_{i,p}^{-1})$.

Put $R(s) = \sum_{p \in \mathcal{P}}\frac{\rho_F(p)}{p^s},$ so that we have \begin{equation}\label{rineq}\log L_{\text{std}}(\pi_{F}, s) =  R(s)  + O(1). \end{equation} Since $L_{\text{std}}(\pi_{F}, s)$ has no pole at 1, it follows from~\eqref{rineq} that $\limsup_{s \rightarrow 1^+} R(s) < \infty.$ Moreover, by Lemma~\ref{lemmaineq}, $\rho_F(p) \ge gc^{\frac{2}{g}} -2g + 1$ for all $p \in S^{F,0}_c$. From Lemma~\ref{oeq} we have that $-2g+1 \le \rho_F(p) \le 2g+1$ for all $p$. Thus by an application of Lemma~\ref{lemmalfunc} with $S = S_c^{F,0}$, $C = 2g-1$, $D = g c^{2/g} - 2 g + 1$ and $E = 2g+1$, we have  $$\overline{\delta}_{\mathrm{Dir}}(S^{F,0}_c) \le (2-\tfrac{1}{g})c^{-\frac{2}{g}}.$$
\end{proof}

\begin{proof}[Proof of Theorem~\ref{maintheorem2}] Consider the function $\log L_{\text{spin}}(\pi_{F}, s)$ on the region $\Re(s)>1$.
 Using Remark~\ref{mum}, we have by an argument similar to that above, $$\log L_{\text{spin}}(\pi_{F}, s) = \sum_{p \in \mathcal{P}}\frac{\mu_F(p)}{p^s} + O(1)$$  Put $T(s) = \sum_{p \in \mathcal{P}}\frac{\mu_F(p)}{p^s}.$ Then, because $L_{\text{spin}}(\pi_{F}, s)$ has no pole at 1, it follows that $\limsup_{s \rightarrow 1^+} T(s) < \infty$. From Lemma~\ref{oeq} we have that $-4\le \mu_F(p) \le 4$ for all $p$. So, by an application of Lemma~\ref{lemmalfunc} with $S = S_c^{F,1}$, $C = 4$, $D = c$ and $E = 4$, we have  $$\overline{\delta}_{\mathrm{Dir}}(S^{F,1}_c) \le \frac{4}{c+4}.$$

\end{proof}

\begin{corollary}\label{corollary}Theorems~\ref{maintheorem1} and~\ref{maintheorem2} remain valid with $\overline{\delta}_{\mathrm{Dir}}$ replaced by $\underline{\delta}_{\mathrm{Nat}}$.
\end{corollary}
\begin{proof}This follows from Lemma~\ref{densityrel}.
\end{proof}

\section{Acknowledgements}I would like to thank Jacob Tsimerman for useful discussions that motivated this work, Paul Nelson for his help with the proof of Lemma~\ref{densityrel} above and for his careful reading of the manuscript, and finally the anonymous referee for some excellent suggestions which significantly improved this paper.

\bibliography{lfunction}

\end{document}